\documentclass{amsart} % in was article
\usepackage{amssymb}
\usepackage{amsmath}
\usepackage{amsfonts}

\begin{document}
%%%%%%%%%%%%%%%%%%%%%%%%%%%%%%%%%%%%%%%%%%%%%%%%%%%%%%%%%%%%%%%%%%%%%%
%	spaces for your own definitions follows
%%%%%%%%%%%%%%%%%%%%%%%%%%%%%%%%%%%%%%%%%%%%%%%%%%%%%%%%%%%%%%%%%%%%%%
\newtheorem{theo}{Theorem}[section]
\newtheorem{prop}[theo]{Proposition}
\newtheorem{lemma}[theo]{Lemma}
\newtheorem{exam}[theo]{Example}
\newtheorem{coro}[theo]{Corollary}
\theoremstyle{definition}
\newtheorem{defi}[theo]{Definition}
\newtheorem{rem}[theo]{Remark}

%\renewcommand{\theequation}{\mbox{\arabic{section}.\arabic{equation}}}

%letters - added these
\newcommand{\Bb}{{\bf B}}
\newcommand{\Nb}{{\bf N}}
\newcommand{\Qb}{{\bf Q}}
\newcommand{\Rb}{{\bf R}}
\newcommand{\Zb}{{\bf Z}}
\newcommand{\Ac}{{\mathcal A}}
\newcommand{\Bc}{{\mathcal B}}
\newcommand{\Cc}{{\mathcal C}}
\newcommand{\Dc}{{\mathcal D}}
\newcommand{\Fc}{{\mathcal F}}
\newcommand{\Ic}{{\mathcal I}}
\newcommand{\Jc}{{\mathcal J}}
\newcommand{\Lc}{{\mathcal L}}
\newcommand{\Oc}{{\mathcal O}}
\newcommand{\Pc}{{\mathcal P}}
\newcommand{\Sc}{{\mathcal S}}
\newcommand{\Tc}{{\mathcal T}}
\newcommand{\Uc}{{\mathcal U}}
\newcommand{\Vc}{{\mathcal V}}

\author{Nik Weaver}

\title [Kadison-Singer problem]
       {The Kadison-Singer problem in discrepancy theory, II}

\address {Department of Mathematics\\
          Washington University in Saint Louis\\
          Saint Louis, MO 63130}

\email {nweaver@math.wustl.edu}

\date{\em March 10, 2013}

\thanks{Supported by NSF grant DMS-1067726.}

%%%%%%%%%%%%%%%%%%%%%%%%%%%%%%%%%%%%%%%%%%%%%%%%%%%%%%%%%%%%%%%%%%%%%%%
%	Please insert the article body now
%%%%%%%%%%%%%%%%%%%%%%%%%%%%%%%%%%%%%%%%%%%%%%%%%%%%%%%%%%%%%%%%%%%%%%%

\begin{abstract}
We apply Srivastava's spectral sparsification technique to a vector
balancing version of the Kadison-Singer problem. The result is a
one-sided version of the conjectured solution.
\end{abstract}

\maketitle

The celebrated Kadison-Singer problem (KSP) is equivalent to a simple
vector balancing question (\cite{W}, Theorem 3). Do there exist
constants $N, r \in {\bf N}$ which
make the following statement true?
\begin{quote}
If $k \in {\bf N}$ and $\{v_1, \ldots, v_m\}$ is a finite sequence of vectors
in ${\bf C}^k$ satisfying $\|v_i\| = \frac{1}{\sqrt{N}}$ for all $i$ and
\[\sum_i |\langle u,v_i\rangle|^2 = 1\]
for all unit vectors $u$, then the index set $\{1, \ldots, m\}$ can be
partitioned into subsets $S_1, \ldots, S_r$ such that
\[\sum_{i \in S_j} |\langle u,v_i\rangle|^2 \leq 1 - \frac{1}{\sqrt{N}}\]
for all unit vectors $u$ and all $j = 1, \ldots, r$.
\end{quote}
Here $\|v\|$ is the euclidean norm of $v$.
It is unclear whether allowing $r > 2$ makes the problem any easier.
If we take $r = 2$ then we can equivalently ask whether it is always
possible to find a subset $S \subseteq \{1, \ldots, m\}$ satisfying
$\frac{1}{\sqrt{N}} \leq
\sum_{i \in S} |\langle u,v_i\rangle|^2 \leq 1 - \frac{1}{\sqrt{N}}$
for all unit vectors $u$. The purpose of this note is to present a
partial positive result in this direction: for any $n < m$
we can find a subset $S \subseteq \{1,\ldots, m\}$ with $|S| = n$ and such that
\[\sum_{i \in S} |\langle u,v_i\rangle|^2 \leq \frac{n}{m} +
O\left(\frac{1}{\sqrt{N}}\right)\]
for all unit vectors $u$. This is a ``one-sided'' version of the desired
result in the sense that we achieve an upper bound but not a lower bound.

Our theorem is a straightforward application of the spectral sparsification
technique introduced in Srivastava's thesis \cite{S}. This technique was
already related to KSP via Bourgain and Tzafriri's
restricted invertibility theorem \cite{SS, S}. That result can be converted
into one resembling ours as in the proof of Theorem 4.2 of \cite{CC}, but
with a substantially worse bound (on the order of
$\frac{n}{m} + 2\sqrt{\frac{n}{m}}$).

\section{The projection version of KSP}
The Kadison-Singer problem was first posed in \cite{KS} in the form
of a C*-algebraic question relating pure states on $B(l^2)$ to pure
states on its diagonal subalgebra. Since then it has been found to have
numerous equivalent versions, and it is now considered a major open
problem with relevance to topics ranging from Banach space theory to
signal processing. We refer to \cite{CC} for general background and a
survey of a variety of equivalent versions of the problem.

Our version is based on the approach of Akemann and Anderson \cite{AA}
in terms of projection matrices. A complex $m \times m$ matrix is a
{\it projection} if the associated linear map orthogonally projects
vectors in ${\bf C}^m$ onto some linear subspace $E$. Note that a
diagonal matrix is a projection if and only if its diagonal entries
are all either zero or one. The Akemann-Anderson
version of KSP asks whether there exist constants $\epsilon, \delta > 0$
and $r \in {\bf N}$ which make the following statement true.
\begin{quote}
If $m \in {\bf N}$ and $P$ is a complex $m \times m$ projection matrix whose
diagonal entries $p_{ii}$ satisfy $p_{ii} \leq \delta$, then there are
diagonal projections $Q_1, \ldots, Q_r$ which sum to the identity matrix
and satisfy $\|Q_jPQ_j\| \leq 1 - \epsilon$ for all $j = 1, \ldots, r$.
\end{quote}
That a positive answer to this question implies a positive solution to
KSP is essentially proven in Propositions 7.6 and 7.7 of \cite{AA},
and the reverse direction is shown in Theorem 1 of \cite{W}. A more
elementary approach to this reduction appears in \cite{B}.

The projection version of KSP is easily seen to be equivalent to a
vector balancing question similar to the one stated in the introduction.
Identify the range $E$ of $P$ with ${\bf C}^k$ where $k$ is the rank of $P$
and define $v_i = Pe_i$, $1 \leq i \leq m$, where $\{e_i\}$ is the standard
basis of ${\bf C}^m$. Then for any vector $u \in {\bf C}^m$ we have
\[\sum |\langle u, v_i\rangle|^2 = \sum |\langle Pu, e_i\rangle|^2
= \|Pu\|^2;\]
in particular, if $u$ is a unit vector in $E$ then this
sum equals 1. Also, $\|v_i\|^2 = \langle Pe_i, e_i\rangle = p_{ii}
\leq \delta$, giving us a bound on the size of the vectors $v_i$.
The diagonal projections $Q_j$ correspond to a partition of $S$ into
$r$ pieces.

The version of KSP stated in the introduction, in which the vectors $v_i$
all have the same norm $\frac{1}{\sqrt{N}}$,
can be achieved by adding extra dimensions
to the space and augmenting the vectors $v_i$ with components in these
extra dimensions. See Theorem 3 of \cite{W} for details. In the language
of projections, this corresponds to requiring that the diagonal entries
of $P$ all equal $\frac{1}{N}$, and asking for $\|Q_jPQ_j\| \leq
1 - \frac{1}{\sqrt{N}}$.

\section{Counterexamples}
Let $k, N \in {\bf N}$ and suppose $\{v_1, \ldots, v_m\}$ is a finite sequence
of vectors in ${\bf C}^k$ satisfying $\|v_i\| = \frac{1}{\sqrt{N}}$ for
$1 \leq i \leq m$ and
\[\sum_i |\langle u,v_i\rangle|^2 = 1\]
for all unit vectors $u$. According to Proposition 4 of \cite{W} we
can find a subset $S \subseteq \{1, \ldots, m\}$ such that
\[\frac{1}{2} - \frac{1}{N} \leq \sum_{i \in S} |\langle e_j, v_i\rangle|^2
\leq \frac{1}{2} + \frac{1}{N}\]
for all $1 \leq j \leq k$, where $\{e_j\}$ is the standard basis of
${\bf C}^k$. This follows by applying the continuous Beck-Fiala theorem
\cite{AA2} to the vectors $(|\langle e_1, v_i\rangle|^2, \ldots,
|\langle e_k,v_i\rangle|^2) \in {\bf R}^k$ for $1 \leq i \leq m$. Thus,
on a fixed orthonormal basis a very tight bound can be achieved. However,
this bound is too strong in general. We know from Example 7 of \cite{W}
that there are configurations of vectors with $N, k \to \infty$
such that for
any $S$ there is some unit vector $u$ for which the sum
$\sum_{i \in S} |\langle u,v_i\rangle|^2$ lies outside
an interval which is asymptotic to $(\frac{1}{2} - \frac{1}{4\sqrt{2N}},
\frac{1}{2} + \frac{1}{4\sqrt{2N}})$.
Thus, the worst counterexamples we know of are $O(\frac{1}{\sqrt{N}})$
away from $\frac{1}{2}$. The obvious conjecture is that we can always find
a set of indices $S$ for which
\[\frac{1}{2} - O\left(\frac{1}{\sqrt{N}}\right)
\leq \sum_{i \in S} |\langle u,v_i\rangle|^2
\leq \frac{1}{2} + O\left(\frac{1}{\sqrt{N}}\right)\]
for all unit vectors $u$. Or perhaps even for any $q \in (0,1)$ we can
always find an $S$ for which
\[q - O\left(\frac{1}{\sqrt{N}}\right)
\leq \sum_{i \in S} |\langle u,v_i\rangle|^2
\leq q + O\left(\frac{1}{\sqrt{N}}\right)\]
for all unit vectors $u$.

What we are trying to accomplish is to build up a set of indices $S$
which makes $\sum_{i \in S} |\langle u,v_i\rangle|^2$ uniformly greater
than $0$ in all directions $u$, while preventing this sum from getting
too close to $1$ in any direction. An example due to Nets Katz \cite{K}
shows why this may be difficult. Fix $N \in {\bf N}$ and let $X$ be the
family of all subsets of $\{1, \ldots, 2N\}$ of size $N$. For each
$1 \leq i \leq 2N$ let $f_i: X \to {\bf R}$ be the function satisfying
$f_i(A) = 0$ if $i \not\in A$ and $f_i(A) = \frac{1}{N}$ if $i \in A$. Then
$\sum_i f_i(A) = |A|\cdot\frac{1}{N} = 1$ for each $A \in X$, so
the functions $f_i$ have sup norm $\frac{1}{N}$ and sum up to $1$ at every
point. Now for any $S \subseteq \{1, \ldots, 2N\}$, if $|S| \leq N$ then we
can find $A \in X$ disjoint from $S$, so that $\sum_{i \in S} f_i(A) = 0$.
But if $|S| \geq N$ then it contains some $A \in X$, and we then have
$\sum_{i \in S} f_i(A) = 1$. So we cannot get away from $0$ at all
points of $X$ without summing to $1$ at some point.

However, the result we prove in the next section shows that nothing like this
can happen with KSP. In Katz's example any set of at least half of
the functions $f_i$ must sum to 1 at some point. Whereas our
theorem achieves an upper bound only $O(\frac{1}{\sqrt{N}})$ higher than
$\frac{|S|}{m}$.

\section{An upper bound}
As in the introduction, $\{v_1, \ldots, v_m\}$ will be a finite set of
vectors in ${\bf C}^k$, each of norm $\frac{1}{\sqrt{N}}$, satisfying
\[\sum_i |\langle u,v_i\rangle|^2 = 1\]
for all unit vectors $u$.
We work with the linear operators $v \otimes v: {\bf C}^k \to {\bf C}^k$
defined by $(v\otimes v)(u) = \langle u,v\rangle v$. For any $S \subseteq
\{1, \ldots, m\}$ the operator $T = \sum_{i \in S} v_i\otimes v_i$ satisfies
\[\langle Tu,u\rangle = \sum_{i \in S} |\langle u,v_i\rangle|^2,\]
and the values $\langle Tu,u\rangle$, for $u$ a unit vector in ${\bf C}^k$,
are all bounded above by $\|T\|$. Thus we are interested in
choosing $S$ so as to minimize $\|T\|$. Note that ${\rm Tr}(v_i\otimes v_i) =
\|v_i\|^2 = \frac{1}{N}$ and $\sum_{i=1}^m v_i \otimes v_i = I$, which has
trace $k$, so that $m = kN$.

Let $n < m$. As in \cite{S}, we build the subset $S$ one vector at a
time. Thus our procedure will select vectors $v_{i_1}, \ldots, v_{i_n}$
with corresponding operators $T_j = \sum_{d=1}^j v_{i_d} \otimes v_{i_d}$.
For any positive operator $T$ and any $a > \|T\|$, define the
{\it upper potential} $\Phi^a(T)$ to be
\[\Phi^a(T) = {\rm Tr}((aI - T)^{-1});\]
then having chosen the vectors $v_{i_1}, \ldots, v_{i_{j-1}}$ we will
select a new vector $v_{i_j}$ so as to minimize $\Phi^{a_j}(T_j)$, where
the $a_j$ are an increasing sequence of upper bounds.
This potential function disproportionately penalizes eigenvalues which are
close to $a_j$ and thereby controls the maximum eigenvalue, i.e., the norm,
of $T_j$. The key fact about the upper potential is given in the following
result.

\begin{lemma}\label{lem1}
(\cite{S}, Lemma 3.4)
Let $T$ be a positive operator on ${\bf C}^k$, let $a, \delta > 0$, and
let $v \in {\bf C}^k$. Suppose $\|T\| < a$. If
\[\frac{\langle ((a + \delta)I - T)^{-2}v,v\rangle}{\Phi^a(T)
- \Phi^{a+\delta}(T)}
+ \langle ((a + \delta)I - T)^{-1}v,v\rangle \leq 1\]
then $\|T + v\otimes v\| < a + \delta$ and
$\Phi^{a + \delta}(T + v\otimes v) \leq \Phi^a(T)$.
\end{lemma}

The proof relies on the Sherman-Morrison formula, which states that if $T$
is positive and invertible then $(T + v\otimes v)^{-1} = T^{-1} -
\frac{T^{-1}(v\otimes v)T^{-1}}{I + \langle T^{-1}v,v\rangle}$.

We also require a simple inequality.

\begin{lemma}\label{lem2}
Let $a_1 \leq \cdots \leq a_k$ and $b_1 \geq \cdots \geq b_k$ be sequences
of positive real numbers, respectively increasing and decreasing. Then
$\sum a_ib_i \leq \frac{1}{k}\sum a_i \sum b_i$.
\end{lemma}

\begin{proof}
Let $M = \frac{1}{k}\sum b_i$. We want to show that $\sum a_i b_i \leq
\sum a_iM$, i.e., that $\sum a_i(b_i - M) \leq 0$. Since the sequence $(b_i)$
is decreasing, we can find $j$ such that $b_i \geq M$ for $i \leq j$ and
$b_i < M$ for $i > j$. Then $\sum_{i=1}^j a_i(b_i - M) \leq a_j\sum_{i=1}^j
(b_i-M)$ (since the $a_i$ are increasing and the values $b_i - M$ are positive)
and $\sum_{i=j+1}^k a_i(b_i-M) \leq a_j\sum_{i=j+1}^k (b_i-M)$
(since the $a_i$ are increasing and the values $b_i - M$ are negative). So
\[\sum_{i=1}^k a_i(b_i-M) \leq a_j \sum_{i=1}^k (b_i-M) = 0,\]
as desired.
\end{proof}

\begin{theo}\label{thm1}
Let $k, N \in {\bf N}$ and let $\{v_1, \ldots, v_m\}$ be a finite sequence
of vectors in ${\bf C}^k$ satisfying $\|v_i\| = \frac{1}{\sqrt{N}}$ for
$1 \leq i \leq m$ and
\[\sum_i |\langle u,v\rangle|^2 = 1\]
for all unit vectors $u$. Then for any $n < m$
there is a set $S \subseteq \{1, \ldots, m\}$ with $|S| = n$ such that
\[\sum_{i \in S} |\langle u,v_i\rangle|^2 \leq \frac{n}{m} +
O\left(\frac{1}{\sqrt{N}}\right)\]
for all unit vectors $u$.
\end{theo}

\begin{proof}
Define $a_i = \frac{1}{\sqrt{N}} +
\left(1 + \frac{1}{\sqrt{N}-1}\right)\frac{i}{m}$ for $0 \leq i \leq n$.
We will find a sequence of distinct indices $i_1, \ldots, i_n$ such that
the operators $T_j = \sum_{d = 1}^j v_{i_d} \otimes v_{i_d}$,
$0 \leq j \leq n$, satisfy $\|T_j\| < a_j$ and $\Phi^{a_0}(T_0) \geq
\cdots \geq \Phi^{a_n}(T_n)$. Thus
\[\|T_n\| < \frac{1}{\sqrt{N}}
+ \left(1 + \frac{1}{\sqrt{N}-1}\right)\frac{n}{m}
= \frac{n}{m} + O\left(\frac{1}{\sqrt{N}}\right),\]
yielding the desired conclusion. We start with $T_0 = 0$, so that
$\Phi^{a_0}(T_0) = \Phi^{1/\sqrt{N}}(0) = {\rm Tr}((\frac{1}{\sqrt{N}}I)^{-1})
= k\sqrt{N}$.

To carry out the induction step, suppose $v_{i_1}, \ldots, v_{i_j}$ have
been chosen. Let $\lambda_1 \leq \cdots \leq \lambda_k$ be the eigenvalues
of $T_j$. Then the eigenvalues of $I - T_j$ are $1-\lambda_1 \geq \cdots
\geq 1 - \lambda_k$ and the eigenvalues of $(a_{j+1}I - T_j)^{-1}$ are
$\frac{1}{a_{j+1} - \lambda_1} \leq \cdots \leq \frac{1}{a_{j+1} - \lambda_k}$.
Thus by Lemma \ref{lem2}
\begin{eqnarray*}
{\rm Tr}((a_{j+1}I - T_j)^{-1}(I - T_j))
&=& \sum_{d=1}^k \frac{1}{a_{j+1} - \lambda_d}(1 - \lambda_d)\cr
&\leq& \frac{1}{k}\sum_{d=1}^k \frac{1}{a_{j+1} - \lambda_d}
\sum_{d=1}^k (1-\lambda_d)\cr
&=& \frac{1}{k}{\rm Tr}((a_{j+1}I - T_j)^{-1}){\rm Tr}(I-T_j)\cr
&=& \frac{1}{k} \Phi^{a_{j+1}}(T_j){\rm Tr}(I - T_j)\cr
&\leq& \frac{1}{k} \Phi^{a_j}(T_j){\rm Tr}(I-T_j)\cr
&\leq& \frac{1}{k} \Phi^{a_0}(T_0){\rm Tr}(I-T_j)\cr
&=& \sqrt{N}{\rm Tr}(I - T_j).
\end{eqnarray*}

Next, $a_{j+1} - a_j = (1 + \frac{1}{\sqrt{N}-1})\frac{1}{m}$, so we can
estimate
\begin{eqnarray*}
\Phi^{a_j}(T_j) - \Phi^{a_{j+1}}(T_j)
&=& {\rm Tr}( (a_jI - T_j)^{-1} - (a_{j+1}I - T_j)^{-1})\cr
&=& \left(1 + \frac{1}{\sqrt{N}-1}\right)\frac{1}{m}
{\rm Tr}( (a_jI - T_j)^{-1}(a_{j+1}I - T_j)^{-1})\cr
&>& \left(1 + \frac{1}{\sqrt{N}-1}\right)\frac{1}{m}
{\rm Tr}((a_{j+1}I - T_j)^{-2})
\end{eqnarray*}
since each of the eigenvalues
$\frac{1}{a_j - \lambda_d}\frac{1}{a_{j+1} - \lambda_d}$ of the operator
$(a_jI - T_j)^{-1}(a_{j+1}I - T_j)^{-1}$ is greater than the corresponding
eigenvalue $\frac{1}{(a_{j+1} - \lambda_d)^2}$ of the operator
$(a_{j+1}I - T_j)^{-2}$. Combining this with Lemma \ref{lem2} yields
\begin{eqnarray*}
{\rm Tr}((a_{j+1}I - T_j)^{-2}(I - T_j))
&\leq& \frac{1}{k}{\rm Tr}((a_{j+1}I - T_j)^{-2}){\rm Tr}(I - T_j)\cr
&<& N\left(1 - \frac{1}{\sqrt{N}}\right)
(\Phi^{a_j}(T_j) - \Phi^{a_{j+1}}(T_j)){\rm Tr}(I-T_j)
\end{eqnarray*}
since $(1 + \frac{1}{\sqrt{N} -1})^{-1} = 1 - \frac{1}{\sqrt{N}}$. Thus
\[\frac{{\rm Tr}((a_{j+1}I - T_j)^{-2}(I-T_j))}{\Phi^{a_j}(T_j) -
\Phi^{a_{j+1}}(T_j)}
\leq N\left(1 - \frac{1}{\sqrt{N}}\right){\rm Tr}(I-T_j).\]

Now let $S' \subseteq \{1, \ldots, m\}$ be the set of indices which
have not yet been used. Observe that
$\langle Tv,v\rangle = {\rm Tr}(T(v\otimes v))$ and that
$\sum_{i \in S'} v_i \otimes v_i = I - \sum_{d = 1}^j v_{i_d}\otimes v_{i_d}
= I - T_j$. Thus
\begin{eqnarray*}
&&\sum_{i \in S'} \left(\frac{\langle (a_{j+1}I -
T_j)^{-2}v_i,v_i\rangle}{\Phi^{a_j}(T_j) - \Phi^{a_{j+1}}(T_j)}
+ \langle (a_{j+1}I - T_j)^{-1}v_i,v_i\rangle\right)\cr
&&= \frac{{\rm Tr}((a_{j+1}I - T_j)^{-2}(I - T_j))}{\Phi^{a_j}(T_j) -
\Phi^{a_{j+1}}(T_j)}
+ {\rm Tr}((a_{j+1}I - T_j)^{-1}(I - T_j))\cr
&&\leq N\left(1 - \frac{1}{\sqrt{N}}\right){\rm Tr}(I-T_j) +
\sqrt{N}{\rm Tr}(I - T_j)\cr
&&= N{\rm Tr}(I - T_j)
\end{eqnarray*}
But
\[N{\rm Tr}(I-T_j) = N(k - {\rm Tr}(T_j)) = m - j\]
is exactly the number of elements of $S'$. So there must exist some $i \in S'$
for which
\[\frac{\langle (a_{j+1}I - T_j)^{-2}v_i,v_i\rangle}{\Phi^{a_j}(T_j)
- \Phi^{a_{j+1}}(T_j)}
+ \langle (a_{j+1}I - T_j)^{-1}v_i,v_i\rangle
\leq 1.\]
Therefore, by Lemma \ref{lem1}, choosing $v_{i_{j+1}} = v_i$ allows the
inductive construction to proceed.
\end{proof}

In terms of projections, Theorem \ref{thm1} states that if $k, N \in {\bf N}$
and $P$ is a projection acting on ${\bf C}^k$ whose diagonal entries $p_{ii}$
all equal $\frac{1}{N}$, then for each $n < kN$ there is a diagonal projection
$Q$ with ${\rm Tr}(Q) = \frac{n}{N}$ and such that $\|QPQ\| \leq
\frac{n}{kN} + O(\frac{1}{\sqrt{N}})$.

\section{A lower bound}
In order to produce a positive solution to the Kadison-Singer problem
we would have to improve Theorem \ref{thm1} to simultaneously include
a lower bound on $\sum_{i \in S} |\langle u,v_i\rangle|^2$. Now
${\rm Tr}(T_n) = \frac{n}{N}$ and $T_n \leq a_nI$ where $a_n =
\frac{n}{m} + O(\frac{1}{\sqrt{N}})$, and thus ${\rm Tr}(a_nI) =
\frac{n}{N} + O(\frac{k}{\sqrt{N}})$. So most of the eigenvalues of
$T_n$ must be around $\frac{n}{m}$. The problem is that there could be
a small fraction of eigenvalues at or near zero.

If we only want a lower bound, the simplest way to achieve this
is to apply Theorem \ref{thm1} and take the operator $I - T_n$. If
$T_n = \sum_{i \in S} v_i\otimes v_i$ then $I - T_n =
\sum_{i \in S^c} v_i \otimes v_i$, so $I - T_n$ is obtained by summing over
$m - n$ vectors. And the upper bound $T_n \leq a_nI$ translates to the
lower bound $I - T_n \geq (1 - a_n)I =
(\frac{m-n}{m} - O(\frac{1}{\sqrt{N}}))I$.
Here the danger is that there could be a small fraction of eigenvalues of
$I - T_n$ at or near one.

If one tries to run the argument of Theorem \ref{thm1} in a way that
simultaneously achieves both upper and lower bounds, one discovers
that the two cases are not really symmetric. At each step the upper bound
recedes, and we need to choose a new vector $v_{i_{j+1}}$ in a way that
avoids overtaking the upper bound. By making the upper bound recede faster,
i.e., by increasing the step size from $a_j$ to $a_{j+1}$, we can ensure
that any desired fraction of the remaining vectors will accomplish this.
The lower bound, on the other hand, is chasing the lower eigenvalues of
$T_j$ and in order to avoid increasing the lower potential we may have
to choose a vector which is concentrated on a possibly small number of
low eigenvalues. Slowing down the lower step size would only delay this.

In order to handle both upper and lower bounds simultaneously, we have to
avoid falling into a situation where the lower bound is approaching a
handful of small eigenvalues, and the only vectors available which have
components among these small eigenvalues also have components
among the largest eigenvalues, and thus cannot be selected without
overtaking the upper bound. It does not seem possible that any greedy
algorithm of the kind used in the proof of Theorem \ref{thm1} could be
sure to prevent such a situation from developing.

%##

\end{document}